\documentclass[preprint, 12pt]{elsarticle}

\usepackage[utf8]{inputenc}
\usepackage{amsmath, amssymb, amsthm}
\usepackage{mathtools}
\usepackage{hyperref}
\usepackage{tcolorbox}
\usepackage{thmtools}
\usepackage[shortlabels]{enumitem}

\numberwithin{equation}{section}

\newtheorem{theorem}{Theorem}[section]
\newtheorem{lemma}[theorem]{Lemma}
\newtheorem{proposition}[theorem]{Proposition}

\theoremstyle{definition}
\newtheorem{definition}[theorem]{Definition}
\newtheorem{remark}[theorem]{Remark}

\hypersetup{
    colorlinks=true,
    linkcolor=blue,
    citecolor=blue,
    urlcolor=blue,
    pdfborder={0 0 0}
}

\DeclarePairedDelimiter{\abs}{\lvert}{\rvert}       % absolute value
\DeclarePairedDelimiter{\n}{\lVert}{\rVert}         % norm
\DeclarePairedDelimiter{\ip}{\langle}{\rangle}      % inner product
\DeclarePairedDelimiter{\set}{\{}{\}}               % set

\newcommand{\newterm}{\emph}                        % new term
\newcommand{\C}{\mathbb{C}}                         % complex
\newcommand{\R}{\mathbb{R}}                         % reals
\newcommand{\N}{\mathbb{N}}                         % naturals
\newcommand{\Hi}{\mathcal{H}}                       % Hilbert space
\newcommand{\B}{{\mathcal{B}(\Hi)}}                 % bounded operators
\newcommand{\Tc}{{\mathcal{S}^1(\Hi)}}              % trace class
\newcommand{\Hs}{{\mathcal{S}^2(\Hi)}}              % Hilbert-Schmidt class
\newcommand{\A}{{\mathcal{A}(\Hi)}}                 % admissible operators
\newcommand{\D}{D^{-1}}                             % inverse duflo moore
\newcommand{\g}{\mathfrak{g}}                       % lie algebra
\newcommand{\z}{\mathfrak{z}}                       % lie algebra center
\newcommand{\Bo}{\mathfrak{B}}                      % Borel sigma algebra
\newcommand{\vp}{\varphi}                           % varphi
\newcommand{\sumn}{\sum_{n \in \N}}                 % sum over naturals n
\newcommand{\limk}{\lim_{k \to \infty}}             % limit as k to infinity
\newcommand{\seqn}[1]{(#1)_{n \in \N}}              % sequence over naturals n
\newcommand{\seqk}[1]{(#1)_{k \in \N}}              % sequence over naturals k
\newcommand{\bek}{\beta_{E_k}}                      % beta function for E_k
\newcommand{\beki}{\beta_{E_k^{-1}}}                % beta function for inverse E_k

\DeclareMathOperator{\dom}{dom}                     % domain
\DeclareMathOperator{\tr}{tr}                       % trace

\newcommand{\proofstep}[2][1]{
    \par\noindent
    \underline{Step #1:}\; \emph{#2}
    \newline
}

\begin{document}

\begin{frontmatter}

\title{Eigenvalue accumulation for operator convolutions on locally compact groups}

\author{Florian Schroth}
\ead{schroth@mathga.rwth-aachen.de}
\affiliation{organization={Chair for Geometry and Analysis, RWTH Aachen University},
            addressline={Pontdriesch 10-12}, 
            city={Aachen},
            postcode={52062}, 
            state={North Rhine-Westphalia},
            country={Germany}}

\begin{abstract}

Within the framework of quantum harmonic analysis, for a locally compact group $G$ with a square-integrable, irreducible unitary representation, we analyze the eigenvalue distributions of convolutions between indicator functions on $G$ and a fixed density operator on the representation space, a concept which generalizes localization operators. In particular, we consider a sequence of such operators and the asymptotic number of eigenvalues that lie within a small distance of $1$. We show that a previously postulated type of asymptotic behavior occurs if and only if the group is unimodular and the sets underlying the indicator functions form a Følner sequence. Applying this, we obtain positive results for nilpotent and homogeneous Lie groups, recovering an established result for the Heisenberg group as a special case.

\end{abstract}


\end{frontmatter}

%%%%%%%%%%%%%%%%%%%%%%%%%%%%%%%%%%%%%%%%%%%%%%%%%%%%%%%%%%%%%%%%%%%%%%%
% START OF CONTENT %
%%%%%%%%%%%%%%%%%%%%%%%%%%%%%%%%%%%%%%%%%%%%%%%%%%%%%%%%%%%%%%%%%%%%%%%

\section{Introduction}\label{sec:intro}

In a 2023 article \cite{Halvdansson}, Simon Halvdansson developed the theory of operator convolutions in the general setting of a locally compact group $G$ with a square-integrable, irreducible unitary representation $\pi$. The conception of such convolutions dates back to Werner \cite{Werner}, who introduced the idea for a projective unitary representation of a symplectic space in the context of theoretical quantum physics. From a purely mathematical standpoint, Berge et al. \cite{QHAAffine} treated the affine group as a first non-unimodular example, shortly before Halvdansson addressed the abstract case. In particular, he gave a detailed analysis of the convolution between a complex-valued function on the group $G$ and an operator on the representation space $\Hi$, and the one between two operators on $\Hi$. The function-operator convolution yields another operator on $\Hi$ and is defined (in the weak sense) by
\[
    f * T := \int_G f(x) \pi(x)^* T \pi(x) d\mu(x).
\]
In the case where $T$ has rank one, the convolution reduces to a localization operator, which is well-studied in the literature (see, e.g., \cite{LocalizationOps3}, \cite{LocalizationOps1}, \cite{LocalizationOps2}). The operator-operator convolution results in a function on $G$ defined via the trace as
\[
    (S * T)(x) := \tr(S \pi(x)^* T \pi(x)).
\]
Most notably, the question of integrability of such a convolution becomes more delicate in the non-unimodular case and \cite{QHAAffine} established that the concept of an admissible operator is required (similar to that of an admissible vector of the representation space).

We focus on one of the many interesting applications described in \cite{Halvdansson}, namely the distribution of the eigenvalues of a convolution $E * S := \chi_E * S$ of a density operator with an indicator function. In particular, if we consider a sequence $\seqk{E_k}$ of subsets of $G$ and the eigenvalues $\seqn{\lambda_n^{(k)}}$ of $E_k * S$, we are interested in the asymptotic behavior of the numbers $\#\set{n \mid \lambda_n^{(k)} > 1-\delta}$, counting those eigenvalues that have a small distance to $1$ (the eigenvalues will always be bounded from above by $1$). Such quantities have been investigated in the case of the Heisenberg group in \cite{MainResHeisenberg1} and \cite{MainResHeisenberg2}, and the results achieved there appear as special cases of our main theorem. Related work has also been done in \cite{DifferentRegionsforEV} for counting eigenvalues inside regions other than $(1-\delta, 1]$, where strict bounds instead of asymptotic behavior were obtained. Halvdansson \cite[Theorem 5.14]{Halvdansson} stated that, in the case of the affine group, $\#\set{n \mid \lambda_n^{(k)} > 1-\delta}$ grows asymptotically like $\tr(S) \mu(E_k)$ if the sets $E_k := \Gamma_{r_k}(E)$ are obtained by scaling a fixed set $E \subseteq G$ using a sequence of dilations
\[
    \Gamma_{r_k}: G \to G, \, (x, a) \mapsto (r_k x, a^{r_k}),
\]
where $\limk r_k = \infty$. We will disprove this claim by showing that, under mild assumptions on the group $G$ and the sequence $\seqk{E_k}$, a limit of the form
\[
    \limk \frac{\#\set{n \mid \lambda_n^{(k)} > 1-\delta}}{\tr(S) \mu(E_k)} = 1
\]
for all $\delta > 0$ is possible if and only if $G$ is unimodular and $\seqk{E_k}$ is a Følner sequence.

Recalling that simply connected, connected, nilpotent Lie groups are unimodular and amenable, we apply the first implication to such groups, where the sets $\seqk{E_k}$ will be growing balls with respect to some word metric. Following a remark in \cite{Halvdansson}, we also consider homogeneous groups (which include stratified Lie groups like the Heisenberg group), where we use dilations of a fixed set for $\seqk{E_k}$. We introduce a slight modification to the definition of the operator convolutions in the nilpotent setting, which is necessary to accommodate an irreducible representation that is only square-integrable modulo the group's center. As an example, we then recover the corresponding result for the Heisenberg group.

The article is organized as follows: In Section \ref{sec:prel}, we will give an overview of the mathematical objects that are of interest here, recall some well-known facts without proof, and introduce the necessary notation. Section \ref{sec:folner} briefly treats Følner sequences in locally compact groups, establishes a few facts that are needed for the main result, and outlines how to construct Følner sequences in nilpotent and homogeneous Lie groups. In Section \ref{sec:opconv}, operator convolutions are defined precisely, we recall some of the properties shown in \cite{Halvdansson}, and supplement what is additionally needed. Further, we present a way to view both kinds of operator convolutions as representation-theoretic objects related to the adjoint representation of $\pi$. Section \ref{sec:main} is concerned with the proof of our main result, Theorem \ref{thm:main}, while Section \ref{sec:appl} describes how to apply it to nilpotent and homogeneous Lie groups.

Note that parts of this article, more precisely, the elaboration on operator convolutions, the first implication of Theorem \ref{thm:main}, and its application to stratified Lie groups, stem from the author's recently completed Master's thesis.

%%%%%%%%%%%%%%%%%%%%%%%%%%%%%%%%%%%%%%%%%%%%%%%%%%%%%%%%%%%%%%%%%%%%%%%
%%%%%%%%%%%%%%%%%%%%%%%%%%%%%%%%%%%%%%%%%%%%%%%%%%%%%%%%%%%%%%%%%%%%%%%

\section{Preliminaries and Notation}\label{sec:prel}

\subsection{Trace class operators on a Hilbert space}\label{sub:traceclass}

Throughout this entire article, $\Hi$ will denote a separable Hilbert space over the field $\C$ of complex numbers. The inner product $\ip{\cdot, \cdot}$ on $\Hi$ and the induced norm $\n{\cdot}$, are written without subscript. $\B$ is the space of bounded operators on $\Hi$ with the operator norm $\n{\cdot}_\B$. If $T \in \B$, then $T^*$ denotes the adjoint operator of $T$ and $\abs{T} := (T^* T)^{1/2}$ its modulus. If $\psi, \vp \in \Hi$, then we define the \newterm{rank-one tensor} $\psi \otimes \vp$ as the operator on $\Hi$ mapping $\xi \mapsto \ip{\xi, \vp} \psi$.

Further interesting for us is the trace class on $\Hi$, which, in a certain sense, will replace $L^1(G)$ when moving from the usual function convolution to operator convolutions. For a more detailed description and proofs of the following facts, see \cite{Schatten}. Fix an orthonormal basis $\seqn{\vp_n}$ of $\Hi$ and define $\Tc$, the \newterm{trace class}, to be the set of all operators $S \in \B$ satisfying
\[
    \n{S}_\Tc := \sumn \ip{\abs{S}\vp_n, \vp_n} < \infty.
\]
$\Tc$ together with the \newterm{trace norm} $\n{\cdot}_\Tc$ is a Banach space and a two-sided norm ideal in $\B$, meaning for every $S \in \Tc$ and $T \in B(\Hi)$, we have $ST, TS \in \Tc$ and $\n{ST}_\Tc, \n{TS}_\Tc \leq \n{T}_\B \n{S}_\Tc$. For every $S \in \Tc$, the sum
\[
    \tr(S) := \sumn \ip{S\vp_n, \vp_n}
\]
is called the \newterm{trace} of $S$. It is absolutely convergent and $\tr$ defines a bounded, linear functional on $\Tc$ with norm $1$. The trace is commutative in the sense that $\tr(TS) = \tr(ST)$ for $S \in \Tc$ and $T \in \B$. The trace class, trace norm, and trace are all independent of the chosen orthonormal basis $\seqn{\vp_n}$.

Every rank-one tensor $\psi \otimes \vp$ is a trace class operator with 
\[
    \n{\psi \otimes \vp}_\Tc = \n{\psi} \n{\vp}, \quad\quad \tr(\psi \otimes \vp) = \ip{\psi, \vp}.
\]
In fact, the span of all such tensors is dense in $\Tc$ with respect to the trace norm, and we have the following singular value decomposition. Here, $\ell^1(\N)$ denotes the space of absolutely summable complex sequences.

\begin{theorem}\label{thm:svd}
    An operator $S \in \B$ belongs to $\Tc$ if and only if there exist orthonormal bases $\seqn{\psi_n}$ and $\seqn{\vp_n}$ and a non-increasing sequence $\seqn{s_n} \in \ell^1(\N)$ of non-negative real numbers such that
    \begin{equation}\label{eq:svd}
        S = \sumn s_n (\psi_n \otimes \vp_n)
    \end{equation}
    with strong operator convergence. In this case, $\seqn{s_n}$ is unique, \eqref{eq:svd} converges even with respect to $\n{\cdot}_\Tc$, and we have $\n{S}_\Tc = \n{\seqn{s_n}}_{\ell^1}$ and
    \[
        \tr(S) = \sumn s_n \ip{\psi_n, \vp_n}.
    \]
    If $S$ is additionally non-negative, then $\seqn{\psi_n} = \seqn{\vp_n}$ and $\seqn{s_n}$ are the eigenvalues of $S$, counted with multiplicity.
\end{theorem}

If $T \in \B$ is normal and $\rho$ is a bounded, complex-valued Borel function defined on the spectrum of $T$, we denote by $\rho(T)$ the bounded operator that is defined by applying $\rho$ to $T$ via the functional calculus. In the special case where $S \in \Tc$ is non-negative with eigenvalues $\seqn{s_n}$, and where $\rho$ is a complex-valued function defined on $\set{s_n}_{n \in \N}$ such that $\seqn{\rho(s_n)} \in \ell^1(\N)$, then $\rho(S) \in \Tc$ and
\begin{equation}\label{eq:tracefctcalc}
    \tr(\rho(S)) = \sumn \rho(s_n).
\end{equation}

%%%%%%%%%%%%%%%%%%%%%%%%%%%%%%%%%%%%%%%%%%%%%%%%%%%%%%%%%%%%%%%%%%%%%%%

\subsection{Abstract harmonic analysis on locally compact groups}\label{sub:aha}

The discussion in Sections \ref{sec:folner}, \ref{sec:opconv}, and \ref{sec:main} will take place on a locally compact group $G$ and, for convenience, we will assume throughout the entire article that $G$ is $\sigma$-compact. The results we recall here are standard and can be found, for example, in \cite{AHA}. We denote the identity element of $G$ by $e$, and for $x \in G$ and subsets $X, Y \subseteq G$, we use the shorthands
\[
    xY := \set{xy \mid y \in Y}, \quad Yx := \set{yx \mid y \in Y}, \quad XY := \set{xy \mid x \in X, y \in Y},
\]
\[
    X^m := \underbrace{X \dots X}_{m\text{-times}}, \quad X^{-1} := \set{x^{-1} \mid x \in X}.
\]
$\Bo(G)$ is the Borel $\sigma$-algebra generated by the open sets of $G$ and $\mu$ will always denote a fixed \newterm{right Haar measure} of $G$, i.e. a Radon measure on $\Bo(G)$ satisfying $\mu(Ex) = \mu(E)$ for all $E \in \Bo(G)$ and all $x \in G$. For $1 \le p \le \infty$, $L^p(G)$ denotes the $p$-Lebesgue space on $G$ with respect to the measure $\mu$, equipped with the usual norm. The \newterm{modular function} of $G$ is defined as
\[
    \Delta_G: G \to \R_{>0}, x \mapsto \frac{\mu(x^{-1} E)}{\mu(E)}
\]
for some $E \in \Bo(G)$ with $0 < \mu(E) < \infty$, where the definition does not depend on $E$. $\Delta_G$ is a continuous group homomorphism and gives the substitution rules
\[
    \int_G f(y) d\mu(y) = \int_G \Delta_G(y) f(y^{-1}) d\mu(y)
\]
and
\[
    \int_G f(y) d\mu(y) = \Delta_G(x) \int_G f(x^{-1} y) d\mu(y)
\]
for $f \in L^1(G)$ and $x \in G$. If $\Delta_G \equiv 1$, we call the group $G$ \newterm{unimodular}.

Essential for the definition of operator convolutions in Section \ref{sec:opconv} is the existence of a square-integrable, irreducible unitary representation of $G$ on some separable Hilbert space $\Hi$, which we denote by $\pi: G \to \mathcal{U}(\Hi)$. For $\psi, \vp \in \Hi$, we call the function
\[
    C_{\psi, \vp}: G \to \C, x \mapsto \ip{\pi(x)\psi, \vp}
\]
the \newterm{matrix coefficient} of $\psi$ and $\vp$ with respect to $\pi$. Then $\pi$ is called \newterm{irreducible} if for all $\psi, \vp \in \Hi$, $C_{\psi, \vp} \equiv 0$ implies $\psi = 0$ or $\vp = 0$, and it is called \newterm{square-integrable} if there exist $\psi, \vp \in \Hi\setminus\set{0}$ such that $C_{\psi, \vp} \in L^2(G)$.

The following classical result about such representations is due to Duflo and Moore \cite{DufloMoore}.

\begin{theorem}\label{thm:duflomoore}
    Let $\pi$ be an irreducible, square-integrable representation of $G$. Then there exists a unique densely defined, self-adjoint, positive operator $D$ on $\Hi$ with densely defined inverse $\D$ such that $C_{\psi, \vp} \in L^2(G)$ if and only if $\psi \in \Hi$ and $\vp \in \dom(\D)$. For all $\psi_1, \psi_2 \in \Hi$ and $\vp_1, \vp_2 \in \dom (\D)$, we have the orthogonality relation
    \[
        \ip{C_{\psi_1, \vp_1}, C_{\psi_2, \vp_2}}_{L^2 (G)} = \ip{\psi_1, \psi_2} \overline{\ip{\D \vp_1, \D \vp_2}}.
    \]
    Further, $\dom(D)$ is invariant under $\pi(x)$ with
    \[
        \pi(x) D = \Delta_G(x)^{-1/2} D \pi(x)
    \]
    on $\dom(D)$ for all $x \in G$.
\end{theorem}

The elements of $\dom(\D)$ are called \newterm{admissible vectors}.

%%%%%%%%%%%%%%%%%%%%%%%%%%%%%%%%%%%%%%%%%%%%%%%%%%%%%%%%%%%%%%%%%%%%%%%

\subsection{Nilpotent Lie groups}\label{sub:nilpotent}

For our application, we are only concerned with finite-dimensional, real Lie groups and the corresponding finite-dimensional, real Lie algebras. For background on nilpotent Lie groups, see \cite[Chapters 1.1, 1.2]{NilpotentLieGroups} and for representations on such groups, \cite[Chapter 4.5]{NilpotentLieGroups}, \cite{SquareintModZ}. Those groups are special because the exponential map $\exp_G: \g \to G$ between a nilpotent Lie algebra $\g$ and the corresponding simply connected, connected, nilpotent Lie group $G$ is a global smooth diffeomorphism and $\lambda \circ \exp_G^{-1}$ (where $\lambda$ is the Lebesgue measure on the Euclidean space $\g$) defines a two-sided Haar measure on $G$, making the group unimodular. Further, nilpotent groups are groups of polynomial growth (see \cite[Théorème II.4]{NilpotentPolynomial}), so \cite[Theorem 1.1]{NilpotentLimit} states that for every compact, symmetric neighborhood $V$ of $e$, there exist constants $Q \in \N$ and $c > 0$ such that
\begin{equation}\label{eq:nilpotentlimit}
    \limk \frac{\mu(V^k)}{c k^Q} = 1.
\end{equation}

In light of Section \ref{sec:opconv}, the subtle problem for our situation is that no non-trivial nilpotent Lie group $G$ admits a square-integrable, irreducible representation (see \cite[Chapter 4.5]{NilpotentLieGroups}). This is essentially a consequence of Schur's lemma and Weil's integral formula, together with the fact that $Z$, the center of $G$, is non-compact. Because the center is prohibitive in this case, it is canonical to weaken the condition of square-integrability to only refer to the integral over $G/Z$ instead of over all of $G$. To make this rigorous, consider the quotient map $q: G \to G/Z$ and a smooth right inverse of $q$, which we denote by $s: G/Z \to G$. Such a map exists and can be constructed as follows: Let $\z \subseteq \g$ be the Lie algebra of $Z$ and let
\[
    \phi: \g/\z \to \g, \, X + \z \mapsto X - P_\z X,
\]
where $P_\z$ is the orthogonal projection onto $\z$ with respect to some arbitrarily chosen inner product on the $\R$-space $\g$. Because $G/Z$ is again a simply connected, connected, nilpotent Lie group and has Lie algebra $\g/\z$, we can define
\[
    s := \exp_G \circ \, \phi \circ \exp_{G/Z}^{-1},
\]
which is a smooth map and satisfies $q \circ s = \operatorname{id}_{G/Z}$. The map $s$ is not unique with these properties, but for our applications this will often not matter. We want to use $s$ to turn a function $f$ defined on $G$ into a function on $G/Z$ via $f \mapsto f \circ s$, and if $f$ is constant on each coset $xZ \subseteq G$, then $f \circ s$ does not depend on the particular choice of $s$. Another useful fact is that $q \circ s = \operatorname{id}_{G/Z}$ implies
\begin{equation}\label{eq:salmosthom}
    s(\dot{x}\dot{y}^{-1}) \in s(\dot{x}) s(\dot{y})^{-1}Z
\end{equation}
for all $x, y \in G$, where $\dot{x}, \dot{y}$ denote the corresponding cosets in $G/Z$. Hence, $s$ is a homomorphism up to translation by $Z$.

With this in place, we call a unitary representation $\pi$ of $G$ on $\Hi$ \newterm{square-integrable modulo $Z$} if there exist $\psi, \vp \in \Hi\setminus\set{0}$ such that $C_{\psi, \vp} \circ s \in L^2(G/Z)$. Note that this definition does not depend on the choice of $s$ since $\abs{C_{\psi, \vp}}$ is constant on each coset. Indeed, some simply connected, connected, nilpotent Lie groups admit representations that are square-integrable modulo $Z$ and irreducible; a characterization is given in \cite{SquareintModZ}. Further, the theorem of Duflo-Moore is adaptable to this situation, and the statement even becomes somewhat simpler, as $G$ is unimodular, making the Duflo-Moore operator a constant multiple of the identity.

\begin{theorem}\label{thm:duflomooremodz}
    If $\pi : G \rightarrow \mathcal{U} (\Hi)$ is an irreducible representation of $G$ that is square-integrable modulo $Z$, then $C_{\psi, \vp} \circ s \in L^2(G/Z)$ for all $\psi, \vp \in \Hi$ and there exists a constant $d > 0$ such that
    \[
        \ip{C_{\psi_1, \vp_1} \circ s, C_{\psi_2, \vp_2} \circ s}_{L^2(G/Z)} = \ip{\psi_1, \psi_2} \overline{\ip{d^{-1} \vp_1, d^{-1} \vp_2}}
    \]
    for all $\psi_1, \psi_2, \vp_1, \vp_2 \in \Hi$.
\end{theorem}

%%%%%%%%%%%%%%%%%%%%%%%%%%%%%%%%%%%%%%%%%%%%%%%%%%%%%%%%%%%%%%%%%%%%%%%

\subsection{Homogeneous groups}\label{sub:homogeneous}

For this subsection, we refer to \cite[Section A]{HomogeneousGroups}. A homogeneous group $G$ is a special type of simply connected, connected, nilpotent Lie group that is interesting to us because of the existence of a family of \newterm{dilations} on $G$ with a set of useful properties. These dilations generalize scalar multiplications on the additive groups $\R^k$ and power functions on the multiplicative group $\R_{>0}$.

\begin{theorem}\label{thm:dilations}
    If $G$ is a homogeneous group, then for every $r > 0$, there exists a map $\Gamma_r: G \to G$ with the following properties.
    \begin{enumerate}[(i)]
        \item For every $r>0$, $\Gamma_r$ is a Lie group automorphism of $G$.
        \item The map $\R_{>0} \to \operatorname{Aut}(G), r \mapsto \Gamma_r$ is a group homomorphism.
        \item For each $x \in G$, the map $\R_{>0} \to G, r \mapsto \Gamma_r(x)$ is continuous with $\lim_{r \downarrow 0} \Gamma_r(x) = e$.
        \item Let $\mu$ denote a Haar measure on $G$. Then there exists $Q \in \N$, called the \newterm{homogeneous dimension} of $G$, such that
        \[
            \mu(\Gamma_r(E)) = r^Q \mu(E)
        \]
        for all $E \in \Bo(G)$ and $r>0$.
    \end{enumerate}
\end{theorem}

%%%%%%%%%%%%%%%%%%%%%%%%%%%%%%%%%%%%%%%%%%%%%%%%%%%%%%%%%%%%%%%%%%%%%%%
%%%%%%%%%%%%%%%%%%%%%%%%%%%%%%%%%%%%%%%%%%%%%%%%%%%%%%%%%%%%%%%%%%%%%%%

\section{Følner sequences}\label{sec:folner}

For a locally compact group $G$, a well-studied property is \newterm{amenability}. There is an abundance of equivalent characterizations (see \cite{Amenability}), however in our context, \newterm{Følner sequences} are most interesting. 

\begin{definition}\label{def:folner}
    \begin{enumerate}[(i)]
        \item For a set $E \in \Bo(G)$ with $0 < \mu(E) < \infty$, define
        \[
            \beta_E: G \to [0, 1], \; x \mapsto 1 - \frac{\mu(E \cap Ex^{-1})}{\mu(E)} = \frac{\mu(E \setminus Ex^{-1})}{\mu(E)}.
        \]
        \item Let $\seqk{E_k} \subseteq \Bo(G)$ be a sequence such that $0 < \mu(E_k) < \infty$ for all $k \in \N$. $\seqk{E_k}$ is called a \newterm{Følner sequence} if $\seqk{\bek}$ converges to $0$ uniformly on every compact subset of $G$.
        \item If $\seqk{E_k} \subseteq \Bo(G)$ with $0 < \mu(E_k) < \infty$ for all $k \in \N$ such that $\seqk{\bek}$ converges to $0$ only pointwise on some set $X \subseteq G$, then we call $\seqk{E_k}$ \newterm{pointwise Følner} on $X$.
    \end{enumerate}
\end{definition}

This notion was first proposed by Følner for discrete groups in \cite{Folner} and it turns out that $G$ is amenable if and only if there exists a Følner sequence in $G$. The first goal of this section is to show that the requirement of locally uniform convergence in Definition \ref{def:folner} (ii) can be substantially weakened.

\begin{proposition}\label{prop:folnerequiv}
    Let $G$ be connected and $\seqk{E_k} \subseteq \Bo(G)$ such that $0 < \mu(E_k) < \infty$ for all $k \in \N$. Then the following three conditions are equivalent.
    \begin{enumerate}[(a)]
        \item $\seqk{E_k}$ is pointwise Følner almost everywhere on a neighborhood of $e$.
        \item $\seqk{E_k}$ is pointwise Følner on $G$.
        \item $\seqk{E_k}$ is a Følner sequence.
    \end{enumerate}
\end{proposition}

For the proof, we need the following two auxiliary statements.

\begin{lemma}\label{la:betaineq}
    Let $E \in \Bo(G)$ with $0 < \mu(E) < \infty$. Then
    \[
        \beta_E(xy) \leq \beta_E(x) + \beta_E(y)
    \]
    for all $x, y \in G$.
\end{lemma}

\begin{proof}
    Note that we have the set inclusion
    \[
        E \setminus E(xy)^{-1} \subseteq (E \setminus Ex^{-1}) \cup (E \setminus Ey^{-1})x^{-1}
    \]
    for all $x, y \in G$. Taking measures, using right invariance, and dividing by $\mu(E_k)$ yields the claim.
\end{proof}

The next lemma follows immediately from Lemma \ref{la:betaineq}.

\begin{lemma}\label{la:folneronprod}
    Let $\seqk{E_k} \subseteq \Bo(G)$ such that $0 < \mu(E_k) < \infty$ for all $k \in \N$. Let $X, Y \subseteq G$ be arbitrary. If $\seqk{E_k}$ is pointwise Følner on $X \cup Y$, then also on $XY$.
\end{lemma}

\begin{proof}[Proof of Proposition \ref{prop:folnerequiv}]
    The implication (c)$\implies$(a) is obvious, so assume (a). Then there exists a compact, symmetric neighborhood $V$ of $e$ and a set $X \subseteq V$ with $\mu(X) = \mu(V)$ such that $\seqk{E_k}$ is pointwise Følner on $X$. Because $V$ is symmetric, we have $X^{-1} \subseteq V$ and since $X$ has full measure in $V$, so does $X^{-1}$. Hence, $X \cap X^{-1}$ is symmetric, has positive measure, and $\seqk{E_k}$ is pointwise Følner on $X \cap X^{-1}$. By Steinhaus's theorem, $U := (X \cap X^{-1})^2$ is then a neighborhood of $e$ and by Lemma \ref{la:folneronprod}, $\seqk{E_k}$ is pointwise Følner on $U$. As $G$ is connected, $G = \bigcup_{m \in \N} U^m$, so an induction applied to Lemma \ref{la:folneronprod} shows that $\seqk{E_k}$ is pointwise Følner on $G$.

    Now assume (b) and let $K \subseteq G$ be compact with positive measure. Using a substitution in Lemma \ref{la:betaineq}, we obtain
    \[
        \bek(x) \le \bek(y) + \bek(xy^{-1})
    \]
    for all $x, y \in K$ and $k \in \N$. Averaging over $K$ yields
    \begin{align}\label{eq:uniformbetaineq}
    \begin{aligned}
        \bek(x) 
            &\leq \mu(K)^{-1} \int_K \bek(y) + \bek(xy^{-1}) d\mu(y) \\
            &= \mu(K)^{-1} \left(\int_K \bek d\mu + \int_{xK^{-1}} \Delta_G(y) \bek(y) d\mu(y)\right) \\
            &\le \mu(K)^{-1} \left(\int_K \bek d\mu + \left(\max_{y \in KK^{-1}} \Delta_G(y)\right) \int_{KK^{-1}} \bek d\mu\right),
    \end{aligned}
    \end{align}
    where we used that the continuous function $\Delta_G$ attains its maximum on the compact set $KK^{-1}$. We know that each $\bek$ is bounded by $1$ and converges pointwise to $0$ because $\seqk{E_k}$ is pointwise Følner. Thus, the dominated convergence theorem states that the right hand side of \eqref{eq:uniformbetaineq} tends to $0$ as $k \to \infty$. Since the expression does not depend on $x \in K$, we have uniform convergence of $\seqk{\bek}$ on $K$.
\end{proof}

\begin{remark}\label{rm:folnerdisconnected}
    If $G$ is not connected in the setting of Proposition \ref{prop:folnerequiv}, then the implication (b)$\implies$(c) still holds, however, in general, (a) implies only that the sequence is pointwise Følner on the identity component $G_0$ of $G$ (or, more generally, the open subgroup of $G$ generated by the given neighborhood of $e$).
\end{remark}

In light of our goal to apply the main theorem to nilpotent and homogeneous Lie groups, we now outline how to construct Følner sequences in these settings. In the general nilpotent case, we consider growing balls with respect to some fixed word metric on $G$ and rely on the fact that $G$ is a group of polynomial growth.

\begin{lemma}\label{la:nilpotentfolner}
    Let $G$ be a nilpotent Lie group and let $V$ be a compact, symmetric neighborhood of $e$ in $G$. Then $\seqk{V^k}$ is a Følner sequence.
\end{lemma}

\begin{proof}
    Let $x \in V$ and note that, for any $k \in \N$, we have $V^{k-1} \subseteq V^k$ and $V^{k-1}x \subseteq V^k$, so that $V^{k-1} \subseteq V^k \cap V^kx^{-1}$. Thus,
    \[
        1 \ge \frac{\mu(V^k \cap V^kx^{-1})}{\mu(V^k)} \ge \frac{\mu(V^{k-1})}{\mu(V^k)} = \frac{\mu(V^{k-1})}{c (k-1)^Q} \, \frac{c k^Q}{\mu(V^k)} \left(\frac{k-1}{k}\right)^Q,
    \]
    and by \eqref{eq:nilpotentlimit}, each term on the right hand side converges to $1$ as $k \to \infty$ if $Q \in \N$ and $c > 0$ are chosen correctly. Since $V$ is a neighborhood of $e$, Proposition \ref{prop:folnerequiv} yields that $\seqk{V^k}$ is a Følner sequence.
\end{proof}

If $G$ is additionally homogeneous, the dilations allow us to scale a fixed set continuously, so easily constructed Følner sequences are even more abundant here.

\begin{lemma}\label{la:homogeneousfolner}
    Let $G$ be a homogeneous group with dilations $(\Gamma_r)_{r>0}$. Let $E \subseteq G$ have finite, non-zero measure, and $\seqk{r_k}$ be a sequence of positive real numbers with $\limk r_k = \infty$. Then $\seqk{\Gamma_{r_k}(E)}$ is a Følner sequence in $G$.
\end{lemma}

\begin{proof}
    The necessary background for this proof is contained in Theorem \ref{thm:dilations}. Using parts (i), (ii), and (iv), we can rewrite
    \[
        \mu(\Gamma_r(E) \cap \Gamma_r(E) x) = \mu (\Gamma_r(E \cap E \Gamma_{r^{-1}}(x))) = r^Q \mu(E \cap E \Gamma_{r^{-1}}(x))
    \]
    for all $r > 0$ and $x \in G$, and
    \[
        \mu(\Gamma_r(E)) = r^Q \mu(E),
    \]
    if $Q$ is the homogeneous dimension of $G$. Part (iii) asserts that $\limk \Gamma_{r_k^{-1}}(x) = e$ for every $x \in G$, so
    \[
        \abs{\mu(E) - \mu(E \cap E \Gamma_{r_k^{-1}}(x))} \le \int_E \abs{\chi_E(y) - \chi_E(y \Gamma_{r_k^{-1}}(x)^{-1})} d\mu(y)
    \]
    converges to $0$ by continuity of right translations in $L^1 (G)$. Thus, combining everything,
    \[
        \limk \frac{\mu(\Gamma_{r_k}(E) \cap \Gamma_{r_k}(E) x)}{\mu(\Gamma_{r_k}(E))} = \limk \frac{\mu(E \cap E \Gamma_{r_k^{-1}}(x))}{\mu(E)} = 1.
    \]
\end{proof}

%%%%%%%%%%%%%%%%%%%%%%%%%%%%%%%%%%%%%%%%%%%%%%%%%%%%%%%%%%%%%%%%%%%%%%%
%%%%%%%%%%%%%%%%%%%%%%%%%%%%%%%%%%%%%%%%%%%%%%%%%%%%%%%%%%%%%%%%%%%%%%%

\section{Operator convolutions}\label{sec:opconv}

Throughout this section and the next one, let $G$ be a locally compact, $\sigma$-compact group and $\pi$ a square-integrable, irreducible representation of $G$ on a separable Hilbert space $\Hi$. The main object of study of this paper is the convolution between a complex-valued function on $G$ and a bounded operator on $G$, which we will consider in the latter part of this section. Implicitly essential is also a type of convolution between a trace class operator and a bounded operator on $\Hi$. Many of the properties we need here can already be found in \cite[Section 3]{Halvdansson}, so we will state them without proof. Whenever something new is needed, a proof will be included. First, we recall the usual convolution between two functions using the right Haar measure.
\begin{definition}\label{def:ffconv}
    Let $f, g \in L^1(G)$. Then, the convolution $f * g: G \to \C$ is defined by
    \[
        (f * g)(x) := \int_G f(y) g(xy^{-1}) d\mu(y)
    \]
    for all $x \in G$.
\end{definition}

Inspired by the necessary translation of the second function in this definition, we define a translation of an operator by an element of $G$ using our unitary representation. Let
\[
    \alpha: G \times \B \to \B, \, (x, T) \mapsto \alpha_x(T) := \pi(x)^* T \pi(x).
\]
Through this and only this will the convolutions depend on the chosen representation, thus we collect the properties of $\alpha$ needed throughout this section in order to argue later that the results may be transferred to the nilpotent Lie group setting. All of the following claims are straightforward to prove.

\begin{proposition}\label{prop:alpha}
    \begin{enumerate}[(i)]
        \item $\alpha$ defines a right group action of $G$ on $\B$, i.e. $\alpha_e(T) = T$ and $\alpha_x(\alpha_y(T)) = \alpha_{yx}(T)$
        for $x, y \in G$ and $T \in \B$.
        \item For each $x \in G$, $\alpha_x$ is a unitary conjugation. In particular, the norms $\n{\cdot}_\B$ and $\n{\cdot}_\Tc$ are invariant under $\alpha_x$.
        \item For each $T \in \B$, the map $G \to \B, \, x \mapsto \alpha_x(T)$ is strongly continuous.
        \item For $x \in G$ and $\psi_1, \psi_2, \vp_1, \vp_2 \in \Hi$, we have
        \[
            \ip{\alpha_x(\vp_2 \otimes \vp_1)\psi_1, \psi_2} = C_{\psi_1, \vp_1}(x) \overline{C_{\psi_2, \vp_2}(x)}.
        \]
    \end{enumerate}
\end{proposition}

%%%%%%%%%%%%%%%%%%%%%%%%%%%%%%%%%%%%%%%%%%%%%%%%%%%%%%%%%%%%%%%%%%%%%%%

\subsection{Operator-operator convolutions}\label{sub:ooconv}

\begin{definition}\label{def:ooconv}
    Let $S \in \Tc$ and $T \in \B$. Then the convolution $S * T: G \to \C$ is defined by
    \[
        (S * T)(x) := \tr(S \alpha_x(T))
    \]
    for all $x \in G$.
\end{definition}

This definition is inspired, in part, by the definition of the usual function convolution. The translation of the second function is replaced by the application of $\alpha_x$ to the second operator, and the integral becomes the trace. This is an important idea coined early on by Segal \cite{Segal}, which permeates this entire section and lets us view many results as analogs of ones for the known function convolution.

We cite \cite[Lemma 3.10, Lemma 3.11]{Halvdansson}.

\begin{lemma}\label{la:ooconvcited}
    Let $S \in \Tc$ and $T \in \B$.
    \begin{enumerate}[(i)]
        \item We have $S * T \in L^\infty(G)$ with $\n{S * T}_{L^\infty(G)} \leq \n{S}_\Tc \n{T}_\B$.
        \item If $S$ and $T$ are non-negative, then so is $S * T$.
    \end{enumerate}
\end{lemma}

Further, we need continuity of the convolution between two operators. This is reminiscent of the fact that the convolution between an integrable function and a bounded one is continuous, if $L^\infty(G)$ is viewed as correspondent to $\B$ and $L^1(G)$ to $\Tc$.

\begin{lemma}\label{la:ooconvcont}
    Let $S \in \Tc$ and $T \in \B$. Then $S * T$ is continuous.
\end{lemma}

\begin{proof}
    Let $S = \sumn s_n (\psi_n \otimes \vp_n)$ be a singular value decomposition of $S$ as described in Theorem \ref{thm:svd}. Because the series
    \[
        \sumn s_n (\psi_n \otimes \vp_n) \alpha_x(T)
    \]
    converges with respect to the trace norm, we can distribute the convolution into the sum, so
    \begin{equation}\label{eq:ooconvdistr}
        (S * T)(x) = \sumn s_n ((\psi_n \otimes \vp_n) * T)(x)
    \end{equation}
    for all $x \in G$. For every $n \in \N$, we have
    \[
        ((\psi_n \otimes \vp_n) * T)(x) = \tr(\alpha_x(T)\psi_n \otimes \vp_n) = \ip{T \pi(x)\psi_n, \pi(x)\vp_n},
    \]
    which is continuous in $x \in G$ because $\pi$ is strongly continuous and $T$ is bounded. Further, by Lemma \ref{la:ooconvcited} (i), the function $(\psi_n \otimes \vp_n) * T$ is bounded by $\n{T}_\B$, thus $S * T$ is continuous because the series \eqref{eq:ooconvdistr} converges uniformly.
\end{proof}

Another important question is that of integrability of the convolution $S * T$. In the unimodular case, and as stated by Werner \cite{Werner} in his first paper on the topic, it is true that the convolution between two trace class operators is in $L^1(G)$. If $G$ is not unimodular however, the problem is more subtle and requires the concept of an \newterm{admissible} operator, introduced in \cite{QHAAffine}. Recall that $D$ denotes the Duflo-Moore operator associated to $\pi$.

\begin{definition}\label{def:adm}
    An operator $A \in \B$ is called \newterm{admissible} if $A$ maps $\dom(D)$ into $\dom(\D)$ and the densely defined operator $\D A \D$ extends to a trace class operator on $\Hi$. By $\A$, we denote the set of all admissible operators on $\Hi$.
\end{definition}

Indeed, this extends the notion of admissibility from elements of the representation space $\Hi$: For any $\psi, \vp \in \Hi\setminus\set{0}$, the operator $\psi \otimes \vp$ is admissible if and only if $\psi, \vp \in \dom(\D)$ (cf. \cite[Proposition 4.7]{Halvdansson}).

The author could not quite follow the proof of the following statement that was presented in \cite{Halvdansson}, so instead, for the sake of completeness, we adapt an argument given for the special case of the affine group in \cite{QHAAffine}. 

\begin{theorem}\label{thm:ooconvl1}
    Let $S \in \Tc$ and $A \in \A$. Then $S * A \in L^1(G)$ with
    \[
        \n{S * A}_{L^1(G)} \le \n{S}_\Tc \n{\D A \D}_\Tc
    \]
    and
    \[
        \int_G S * A \, d\mu = \tr(S) \tr(\D A \D).
    \]
\end{theorem}

\begin{proof}
    Let $R \in \Tc$ be the extension of $\D A \D$ and $R = \sumn r_n (\psi_n \otimes \vp_n)$ a singular value decomposition. We proceed in three steps, considering increasingly general trace class operators for $S$, starting with the simplest case.

    \proofstep[1]{Show the claim for $S = \xi \otimes \eta$, where $\xi, \eta \in \dom(D)$.}
    For fixed $x \in G$, we have
    \begin{align*}
        (S * A)(x) 
            &= \tr((\xi \otimes \eta) \alpha_x(A)) \\
            &= \ip{\alpha_x (A) \xi, \eta}\\
            &= \ip{A \pi (x) \xi, \pi (x) \eta}.
    \end{align*}
    By the second part of Theorem \ref{thm:duflomoore} (Duflo-Moore), $\pi(x) \xi$ and $\pi(x) \eta$ are elements of $\dom(D)$, so here we can write $A = D R D$ and
    \begin{align}\label{eq:ooconvpointwise}
        (S * A)(x) 
            &= \ip{R D \pi (x) \xi, D \pi (x) \eta} \nonumber \\
            &= \Delta_G(x) \ip{R \pi (x) D \xi, \pi (x) D \eta} \nonumber \\
            &= \Delta_G(x) \sumn r_n \ip{(\psi_n \otimes \vp_n) \pi (x) D \xi, \pi (x) D \eta} \nonumber \\
            &= \Delta_G(x) \sumn r_n \ip{\alpha_{x^{-1}}(D\xi \otimes D\eta)\psi_n, \vp_n}.
    \end{align}
    We consider the $L^1$-norm of one term of the last sum and substitute the inverse using the modular function. Part (iv) of Proposition \ref{prop:alpha} then yields
    \begin{align*}
        &\int_G \abs{\Delta_G (x) \ip{\alpha_{x^{-1}}(D\xi \otimes D\eta)\psi_n, \vp_n}} d\mu(x) \\
        = \, &\int_G \abs{\ip{\alpha_x(D\xi \otimes D\eta)\psi_n, \vp_n}} d\mu (x) \\
        = \, &\ip{\abs{C_{\psi_n, D \eta}}, \abs{C_{\vp_n, D \xi}}}_{L^2(G)} \\
        \leq \, &\n{C_{\psi_n, D \eta}}_{L^2(G)} \n{C_{\vp_n, D \xi}}_{L^2(G)} \\
        = \, &\n{\psi_n} \n{\D D \eta} \n{\vp_n} \n{\D D \eta} \\
        = \, &\n{\xi} \n{\eta},
    \end{align*}
    again by Theorem \ref{thm:duflomoore}. By summing over all $n \in \N$, we then obtain
    \[
        \n{S * A}_{L^1(G)} \le \sumn r_n \n{\xi} \n{\eta} = \n{\xi \otimes \eta}_\Tc \n{R}_\Tc.
    \]
    For the integral identity, we simply integrate \eqref{eq:ooconvpointwise}, where the dominated convergence theorem together with the established integrability lets us interchange summation and integration. The same substitution as above and another application of Theorem \ref{thm:duflomoore} then yield
    \begin{align*}
        \int_G S * A \, d\mu 
            &= \int_G \Delta_G(x) \sumn r_n \ip{\alpha_{x^{-1}}(D\xi \otimes D\eta)\psi_n, \vp_n} d\mu(x) \\
            &= \int_G \sumn r_n C_{\psi_n, D \eta}(x) \overline{C_{\vp_n, D \xi}(x)} d\mu(x) \\
            &= \sumn r_n \ip{C_{\psi_n, D \eta}, C_{\vp_n, D \xi}}_{L^2(G)} \\
            &= \sumn r_n \ip{\psi_n, \vp_n} \ip{\D D \xi, \D D \eta} \\
            &= \tr (\xi \otimes \eta) \tr(R).
    \end{align*}

    \proofstep[2]{Show the claim for $S = \xi \otimes \eta$ where $\xi, \eta \in \Hi$ are arbitrary.}
    Since $\dom(D)$ is dense in $\Hi$, we can find sequences $\seqk{\xi_k}, \seqk{\eta_k} \subseteq \dom(D)$ such that $\limk \xi_k = \xi$ and $\limk \eta_k = \eta$. Then $\seqk{\xi_k \otimes \eta_k}$ converges to $\xi \otimes \eta$ in the trace norm, because the bilinear map
    \[
        \Hi \times \Hi \rightarrow \Tc, (\psi, \vp) \mapsto \psi \otimes \vp
    \]
    is bounded. Since operator-operator convolutions are continuous by Lemma \ref{la:ooconvcont}, Lemma \ref{la:ooconvcited} (i) then implies
    \[
        \limk (\xi_k \otimes \eta_k)  * A = (\xi \otimes \eta) * A
    \]
    uniformly. This sequence is further a Cauchy sequence in $L^1(G)$, as
    \begin{align*}
        \n{(\xi_k \otimes \eta_k) * A - (\xi_l \otimes \eta_l) * A}_{L^1(G)} 
            &\leq \n{(\xi_k \otimes (\eta_k - \eta_l)) * A}_{L^1(G)} + \n{((\xi_k - \xi_l) \otimes \eta_l) * A}_{L^1G)} \\
            &\leq (\n{\xi_k} \n{\eta_k - \eta_l} + \n{\xi_k - \xi_l} \n{\eta_l}) \n{R}_\Tc
    \end{align*}
    for all $k, l \in \N$ by step 1. Hence $(\xi \otimes \eta) * A \in L^1 (G)$, and both the claimed norm inequality and the integral identity follow from step 1 applied to each $(\xi_k \otimes \eta_k) * A$, and from $L^1$-convergence of $\seqk{(\xi_k \otimes \eta_k) * A}$ to $(\xi \otimes \eta) * A$.

    \proofstep[3]{Show the claim for arbitrary $S \in \Tc$.}
    Let $S = \sumn s_n (\xi_n \otimes \eta_n)$ be a singular value decomposition. As in the proof of Lemma \ref{la:ooconvcont}, we have
    \begin{equation}\label{eq:ooconvpointwise2}
        (S * A)(x) = \sumn s_n ((\xi_n \otimes \eta_n) * A)(x)
    \end{equation}
    for all $x \in G$. Then, immediately,
    \[
        \n{S * A}_{L^1(G)} \leq \sumn s_n \n{(\xi_n \otimes \eta_n) * A}_{L^1(G)} \leq \sumn s_n \n{\xi_n} \n{\eta_n} \n{R}_\Tc = \n{S}_\Tc \n{R}_\Tc
    \]
    by step 2. With this, we can integrate \eqref{eq:ooconvpointwise2} and apply the dominated convergence theorem to obtain
    \begin{align*}
        \int_G S * A \, d\mu(x) 
            &= \sumn s_n \int_G (\psi_n \otimes \vp_n) * A \, d\mu(x) \\
            &= \sumn s_n \ip{\psi_n, \vp_n} \tr(R) \\
            &= \tr(S) \tr(R)
    \end{align*}
    by the integral identity from step 2.
\end{proof}

%%%%%%%%%%%%%%%%%%%%%%%%%%%%%%%%%%%%%%%%%%%%%%%%%%%%%%%%%%%%%%%%%%%%%%%

\subsection{Function-operator convolutions}\label{sub:foconv}

For the convolution between a function and an operator, we consider two cases.

\begin{definition}\label{def:foconv1}
    Let $f \in L^1(G)$ and $T \in \B$. Then the convolution $f * T: \Hi \to \Hi$ is defined as the unique bounded operator that satisfies
    \[
        \ip{(f * T)\psi, \vp} = \int_G f(x) \ip{\alpha_x(T)\psi, \vp} d\mu(x)
    \]
    for all $\psi, \vp \in \Hi$.
\end{definition}

This definition works because the right hand side is bilinear in $(\psi, \vp)$ and bounded from above by $\n{f}_{L^1(G)} \n{T}_\B \n{\psi} \n{\vp}$. The same bound also shows that 
\[
    \n{f * T}_\B \le \n{f}_{L^1(G)} \n{T}_\B.
\]

\begin{definition}\label{def:foconv2}
    Let $g \in L^\infty(G)$ and $A \in \A$. Then the convolution $g * A$ is defined as in Definition \ref{def:foconv1}.
\end{definition}

In this case, well-definedness and the bound
\begin{equation}\label{eq:foconv2bound}
    \n{g * A}_\B \le \n{g}_{L^\infty(G)} \n{\D A \D}_\Tc
\end{equation}
were shown in \cite[Lemma 4.12]{Halvdansson} and the proof is essentially an application of Theorem \ref{thm:ooconvl1}. The following is contained in \cite[Proposition 3.4, Proposition 3.5, Lemma 3.6, Proposition 3.12]{Halvdansson}.

\begin{lemma}\label{la:foconvcited}
    \begin{enumerate}[(i)]
        \item If $f \in L^1(G)$ and $S \in \Tc$, then $f * S \in \Tc$ with
        \[
            \n{f * S}_\Tc \le \n{f}_{L^1(G)} \n{S}_\Tc
        \]
        and
        \[
            \tr(f * S) = \tr(S) \int_G f d\mu.
        \]
        \item If $f \in L^1(G)$ and $T \in \B$ are both non-negative, then $f * T$ is non-negative.
        \item Let $f, g \in L^1(G)$, $S \in \Tc$, and $T \in \B$. Then
        \[
            (f * S) * T = f * (S * T)
        \]
        and
        \[
            (f * g) * T = f * (g * T).
        \]
    \end{enumerate}
\end{lemma}

In particular, the associativity given by (iii) shows the interplay between all three kinds of convolutions and suggests that the new definitions for the operator convolutions are appropriate. Similar statements to (ii) and (iii) can be formulated for the setting of Definition \ref{def:foconv2} without much additional work, however we restrict ourselves to the material needed for our main result.

We briefly want to outline a connection between the operator convolutions discussed in this section and the adjoint representation of $\pi$.

\begin{remark}\label{rm:convwithadjrepr}
    We denote by $\Hs$ the \newterm{Hilbert-Schmidt class} of operators on $\Hi$ and this space is equipped with the inner product $\ip{S, T}_\Hs := \tr(S T^*)$ for $S, T \in \Hs$, making it a Hilbert space (for details on $\Hs$ and its relation to the trace class, see \cite{Schatten}). The \newterm{adjoint representation} of $\pi$ is a unitary representation of $G$ on $\Hs$ defined by
    \[
        \operatorname{Ad}_\pi : G \to \mathcal{U}(\Hs), \, \operatorname{Ad}_\pi(x)T := \pi(x) T \pi(x)^*.
    \]
    $\operatorname{Ad}_\pi$ is also sometimes written as the tensor product $\pi \otimes \overline{\pi}$.

    If $S, T \in \Hs$, then their product is trace class, so we can define a convolution $S * T$ analogously to Definition \ref{def:ooconv}. But then
    \[
        (S * T)(x) = \tr (S \pi (x)^* T \pi (x)) = \tr ( \operatorname{Ad}_\pi (x)(S) T) = \ip{\operatorname{Ad}_\pi (x) S, T^*}_{\Hs},
    \]
    which is the matrix coefficient of $S$ and $T^*$ with respect to the adjoint representation. Thus, every operator $A \in \Hs \cap \A$ is admissible for the representation $\operatorname{Ad}_\pi$ in the classical sense, but the converse is false in general.

    Further, if $f \in L^1(G)$ and $S \in \Hs$, we can consider the integrated representation of $\operatorname{Ad}_\pi$, which represents $L^1(G)$ in $\B(\Hs)$. By definition, we have
    \begin{align*}
        \ip{\operatorname{Ad}_\pi (f) (S) \psi, \vp} 
            &= \tr (\operatorname{Ad}_\pi (f) (S) ( \psi \otimes \vp )) \\
            &= \ip{\operatorname{Ad}_\pi (f) S, \vp \otimes \psi}_{\Hs} \\
            &= \int_G f(x) \ip{\operatorname{Ad}_\pi (x^{-1} ) S, \vp \otimes \psi}_{\Hs} \, d\mu (x) \\
            &= \int_G f(x) \tr ( \alpha_x (S) ( \psi \otimes \vp )) d\mu (x) \\
            &= \int_G f(x) \ip{\alpha_x (S) \psi, \vp} d\mu (x) \\
            &= \ip{(f * S) \psi, \vp}
    \end{align*}
    for all $\psi, \vp \in H$, so $f * S = \operatorname{Ad}_\pi(f)S$. For the unimodular case, this was also noticed in \cite[Section A.4]{OoconvasIntegrated}.
    
    With these observations, some of the results from this section can be obtained by applying known facts about matrix coefficients and integrated representations of unitary representations, however this limits us to integrable functions and Hilbert-Schmidt operators, and these classes are, in general, too restrictive.
\end{remark}

Keeping in mind that later, we want to apply our main result to nilpotent Lie groups, where the representation is only square-integrable over the quotient with respect to the center, we note the following.

\begin{remark}\label{rm:convmodz}
    Let $G$ be a simply connected, connected, nilpotent Lie group and $\pi$ an irreducible representation, square-integrable modulo its center $Z$. We claim that it is possible to transfer every statement mentioned in this section by replacing $G$ with $G/Z$. To see this, let $\chi: G \to \mathbb{T}$ be the character that determines the action of $\pi|_Z$, fix a smooth right inverse $s$ of the quotient map as described in Section \ref{sub:nilpotent}, and define
    \[
        \dot{\alpha}: G/Z \times \B \to \B, \, (\dot{x}, T) \mapsto \dot{\alpha}_{\dot{x}}(T) := \alpha_{s(\dot{x})}(T) = \pi(s(\dot{x}))^* T \pi(s(\dot{x})).
    \]
    Because
    \begin{equation}\label{eq:alphainv}
        \alpha_{xz}(T) = \pi(xz)^* T \pi(xz) = \overline{\chi(z)} \chi(z) \pi(x)^* T \pi(x) = \alpha_x(T)
    \end{equation}
    for all $x \in G$, $z \in Z$, we see that $\alpha$ is invariant under translation by $Z$, so $\dot{\alpha}$ does not depend on the specific choice of the section $s$. Further, $\dot{\alpha}$ has all properties described for $\alpha$ in Proposition \ref{prop:alpha}: (i) follows from the corresponding fact about $\alpha$ together with \eqref{eq:salmosthom} and \eqref{eq:alphainv}, (ii) is trivial, (iii) follows from continuity of $s$, and (iv) is true if we replace the matrix coefficients by their compositions with $s$ -- exactly what is necessary for Theorem \ref{thm:duflomooremodz}, and therefore for the proof of Theorem \ref{thm:ooconvl1} to remain intact. This allows us to define operator convolutions with respect to $\pi$ using $\dot{\alpha}$ instead of $\alpha$ and all results carry over.
\end{remark}

%%%%%%%%%%%%%%%%%%%%%%%%%%%%%%%%%%%%%%%%%%%%%%%%%%%%%%%%%%%%%%%%%%%%%%%
%%%%%%%%%%%%%%%%%%%%%%%%%%%%%%%%%%%%%%%%%%%%%%%%%%%%%%%%%%%%%%%%%%%%%%%

\section{Main result}\label{sec:main}

As outlined in the introduction, we now want to analyze the eigenvalues of a sequence $\seqk{E_k * S}$ of function-operator convolutions, where $E * S$ is short hand for the convolution with the indicator function $\chi_E * S$. For each $k \in \N$, $E_k$ will be a measurable subset of $G$ with $0 < \mu(E_k) < \infty$, and $S$ a density operator, defined as follows.

\begin{definition}\label{def:densityop}
    A non-negative operator $S \in \Tc \cap \A$ that satisfies $\tr(\D S \D) = 1$ is called a \newterm{density operator}.
\end{definition}

Thus, a density operator combines admissibility, which allows for integrability of operator-operator convolutions by Theorem \ref{thm:ooconvl1}, with membership in the trace class, which makes the convolution with an integrable function again trace class by Lemma \ref{la:foconvcited} (i). It is noteworthy that our arguments still work if $\tr (\D S \D )$ is only assumed to be positive instead of equal to $1$, and this would yield that the eigenvalues accumulate below $\tr ( \D S \D )$ instead of below $1$. We stick with Definition \ref{def:densityop}, though, for ease of argument and because historically, density operators have been defined in this way (see e.g., \cite{DensityOps}).

The next Lemma is the basis of one direction of Theorem \ref{thm:main}.

\begin{lemma}\label{la:boundtrrhot}
    Let $T \in \Tc$ be non-negative and non-zero with $\n{T}_\B \leq 1$ and with eigenvalues $\seqn{\lambda_n}$. Further let $\rho: [0, 1] \to \R$ be a bounded Borel function that satisfies $\abs{\rho(t)} \le t$ for all $t \in [0, 1]$. Then $\rho(T)$ is again trace class and
    \[
        \frac{\tr(\rho(T))}{\tr(T)} \leq 1 - \frac{\#\set{n \mid \lambda_n > 1-\delta}}{\tr(T)} (1 - \delta - \sup_{1-\delta < t \le 1} \rho(t))
    \]
    for every $\delta > 0$.
\end{lemma}

\begin{proof}
    Since $T$ is non-negative with $\n{T}_\B \le 1$, all eigenvalues lie in $[0, 1]$, so $\rho(T)$ is well-defined via the functional calculus. By the assumption $\abs{\rho(t)} \le t$ for all $t \in [0, 1]$, $\rho$ clearly preserves summability of sequences in $[0, 1]$, so $\rho(T) \in \Tc$. Now consider the decomposition
    \[
        \tr(\rho(T)) = \sum_{\lambda_n \le 1-\delta} \rho(\lambda_n) + \sum_{\lambda_n > 1-\delta} \rho(\lambda_n).
    \]
    For the first term, we use our assumption on $\rho$ again to obtain
    \begin{align*}
        \sum_{\lambda_n \le 1-\delta} \rho(\lambda_n) 
            &\le \sum_{\lambda_n \leq 1-\delta} \lambda_n \\
            &= \tr(T) - \sum_{\lambda_n > 1-\delta} \lambda_n \vphantom{\int}\\
            &\le \tr(T) - \#\set{n \mid \lambda_n > 1-\delta} (1 - \delta) \vphantom{\int},
    \end{align*}
    and the second term is bounded by
    \[
        \sum_{\lambda_n > 1-\delta} \rho(\lambda_n) \leq \#\set{n \mid \lambda_n > 1-\delta} \sup_{1-\delta < t \le 1} \rho(t).
    \]
    Adding up and dividing by $\tr(T) \ne 0$ yields the claim.
\end{proof}

We need one more technical step before tackling our main result, and this can partly be found in \cite[Lemma 5.16, Lemma 5.18]{Halvdansson}

\begin{lemma}\label{la:traceofsquare}
    Let $S$ be a density operator. 
    \begin{enumerate}[(i)]
        \item The function $h_S := \tr(S)^{-1} (S * S)$ is well defined, continuous, non-negative, symmetric, and integrable with total integral $1$, and it satisfies $h_S(e) > 0$.
        \item For any $E \in \mathfrak{B}(G)$ with $\mu(E) < \infty$, we have
        \[
            \tr((E * S)^2) = \tr(S) \int_G h_S(x) \mu(E \cap x^{-1}E) d\mu(x)
        \]
    \end{enumerate}
\end{lemma}

\begin{proof}
    The density operator $S$ is non-negative, non-zero, and trace class, so $\tr(S) > 0$. Thus, $h_S$ is well-defined, non-negative, and continuous by Lemmas \ref{la:ooconvcited} (ii) and \ref{la:ooconvcont}. It is symmetric because
    \[
        (S * S)(x^{-1}) = \tr(S \pi(x)S\pi(x)^*) = \tr(S \pi(x)^*S\pi(x))
    \]
    for all $x \in G$ by commutativity of the trace, $\int_G h_S d\mu = 1$ follows from Theorem \ref{thm:ooconvl1}, and the observation $h_S(e) = \tr(S^2) > 0$ concludes part (i).

    For part (ii), we begin by taking $y \in G$ and an orthonormal basis $\seqn{\vp_n}$ of $\Hi$, and computing
    \begin{align*}
        (S * (E * S))(y) 
            &= \tr(S \alpha_y(E * S)) \vphantom{\sumn}\\
            &= \sumn \ip{S \alpha_y(E * S) \vp_n, \vp_n} \\
            &=\sumn \int_E \ip{\alpha_x(S) \pi(y) \vp_n, \pi(y) S \vp_n} d\mu(x) \\
            &= \int_E \sumn \ip{S \alpha_{xy}(S) \vp_n, \vp_n} d\mu(x) \\
            &= \int_E \tr(S \alpha_{xy}(S)) d\mu(x) \vphantom{\sumn}\\
            &= \tr(S) \int_E h_S(xy) d\mu(x),
    \end{align*}
    where interchanging the sum and the integral is justified by the dominated convergence theorem because the integrand is bounded by $\n{S}_\B \n{S}_\Tc$ and $E$ has finite measure. Using the associativity from Lemma \ref{la:foconvcited} (iii),
    \begin{align*}
        \tr((E * S)^2) 
            &= ((E * S) * (E * S))(e) \\
            &= (\chi_E * (S * (E * S)))(e) \\
            &= \tr(S) \left(\chi_E * \left(\int_E h_S(x(\cdot)) d\mu(x)\right)\right)(e) \\
            &= \tr(S) \int_E \int_E h_S(xy^{-1}) d\mu(x) d\mu(y).
    \end{align*}
    Right invariance of the inner integral and an application of Fubini's theorem then yield
    \begin{align*}
        \tr((E * S)^2) 
            &= \tr(S) \int_E \int_G \chi_E(xy) h_S(x) d\mu(x) d\mu(y) \\
            &= \tr(S) \int_G h_S(x) \int_E \chi_E(xy) d\mu(y) d\mu(x) \\
            &= \tr(S) \int_G h_S(x) \mu(E \cap x^{-1}E) d\mu(x).
    \end{align*}
\end{proof}

Now we can finally prove the relation between the asymptotic behavior of the eigenvalues of the operator sequence $\seqk{(E_k * S)}$ and fundamental properties of the underlying group.

\begin{theorem}\label{thm:main}
    Let $G$ be a locally compact, connected group with a square-integrable, irreducible representation $\pi$ on a separable Hilbert space $\Hi$. Let $S$ be a density operator on $\Hi$ and $\seqk{E_k} \subseteq \mathfrak{B}(G)$ a sequence of sets such that $0 < \mu(E_k) < \infty$ for all $k \in \N$. Then, for each $k \in \N$, $E_k * S$ is non-negative and trace class with eigenvalues $\seqn{\lambda_n^{(k)}} \subseteq [0, 1]$. Further, the following two statements are equivalent.
    \begin{enumerate}[(a)]
        \item $G$ is unimodular and $\seqk{E_k^{-1}}$ is a Følner sequence.
        \item For every $\delta \in (0, 1)$, we have
        \[
            \limk \frac{\#\set{n \mid \lambda_n^{(k)} > 1-\delta}}{\tr(S) \mu(E_k)} = 1.
        \]
    \end{enumerate}
\end{theorem}

The proof of the first implication is essentially an application of an idea by Halvdansson \cite[Lemma 5.19]{Halvdansson}.

\begin{proof}
    First of all, we denote 
    \[
        C_\delta^{(k)} := \frac{\#\set{n \mid \lambda_n^{(k)} > 1-\delta}}{\tr(S) \mu(E_k)}
    \]
    for $\delta \in (0, 1)$ and $k \in \N$. The properties of each individual $E_k * S$ follow from Lemma \ref{la:foconvcited} and the bound \eqref{eq:foconv2bound}, which takes the form
    \[
        \n{E_k * S}_\B \leq \n{\chi_{E_k}}_{L^\infty(G)} \n{\D S \D}_\Tc \leq 1
    \]
    in this case because $S$ is a density operator. Now assume (a), fix $\delta \in (0, 1)$ and $k \in \N$, and consider the function
    \[
        \theta : [0, 1] \rightarrow [ \delta - 1, \delta), \; t \mapsto
        \begin{cases}
            -t, \quad& \text{if} \; \, 0 \leq t \leq 1 - \delta \\
            1 - t, & \text{if} \; \, 1 - \delta < t \leq 1
        \end{cases}.
    \]
    $\theta$ is bounded, Borel-measurable, and preserves summability of sequences in $[0, 1]$ since the terms of any $\ell^1$-sequence are eventually in $[0, 1-\delta]$, where we have $\abs{\theta(t)} = t$. Thus, the operator $\theta(E_k * S)$ is well-defined and trace class, and we have
    \[
        \tr(\theta(E_k * S)) = \sumn \theta(\lambda_n^{(k)}) = \#\set{n \mid \lambda_n^{(k)} > 1-\delta} - \sumn \lambda_n^{(k)},
    \]
    where
    \[
        \sumn \lambda_n^{(k)} = \tr (E_k * S) = \mu(E_k) \tr(S)
    \]
    by Lemma \ref{la:foconvcited} (i). As a polynomial upper bound for $\abs{\theta}$, we consider
    \[
        \sigma : [0, 1] \rightarrow \R, \; t \mapsto m_\delta \, t(1-t),
    \]
    where $m_\delta := \max\set{\delta^{-1}, (1-\delta)^{-1}}$. Indeed, for $t \in [0, 1-\delta]$,
    \[
        \abs{\theta(t)} = t \leq t \cdot \frac{1-t}{\delta} \leq \sigma(t),
    \]
    and for $t \in (1 - \delta, 1]$,
    \[
        \abs{\theta(t)} = 1-t \leq (1-t) \frac{t}{1-\delta} \leq \sigma(t).
    \]
    Thus, $\abs{\theta} \leq \sigma$, and we have
    \begin{align*}
        &\abs{\#\set{n \mid \lambda_n^{(k)} > 1-\delta} - \mu(E_k) \tr(S)} \\
        = \, &\abs{\tr (\theta(E_k * S))} \vphantom{\int_G} \\
        \le \, &\tr(\abs{\theta}(E_k * S)) \\
        \le \, &\tr(\sigma(E_k * S)) \vphantom{\int_G} \\
        = \, &m_\delta \left(\tr(E_k * S) - \tr((E_k * S)^2)\right) \\
        = \, &m_\delta \left(\mu(E_k) \tr(S) - \tr(S) \int_G h_S(x) \mu(E_k \cap x^{-1}E_k) d\mu(x)\right),
    \end{align*}
    where we used Lemma \ref{la:traceofsquare} (ii) in the last step. Dividing by $\tr(S) \mu(E_k)$ yields
    \begin{equation}\label{eq:boundc}
        \abs{C_\delta^{(k)} - 1} \leq m_\delta \left(1 - \int_G h_S(x) \frac{\mu(E_k \cap x^{-1}E_k)}{\mu(E_k)} d\mu(x)\right).
    \end{equation}
    Because $G$ is unimodular, the Haar measure is invariant under inversion, so
    \[
        \frac{\mu(E_k \cap x^{-1}E_k)}{\mu(E_k)} = \frac{\mu(E_k^{-1} \cap E_k^{-1}x)}{\mu(E_k^{-1})},
    \]
    and the right hand side is bounded by $1$ and converges pointwise to $1$ because $\seqk{E_k^{-1}}$ is a Følner sequence. Since $h_S$ is integrable with total integral $1$ by Lemma \ref{la:traceofsquare} (i), the dominated convergence theorem implies
    \[
        \limk \int_G h_S(x) \frac{\mu(E_k \cap x^{-1}E_k)}{\mu(E_k)} d\mu(x) = \int_G h_S(s) d\mu (x) = 1.
    \]
    Together with \eqref{eq:boundc}, we obtain (b).

    Now, we assume (b) and apply Lemma \ref{la:boundtrrhot} to each operator $E_k * S$ and the function $\rho: [0, 1] \to \R, t \mapsto t(1-t)$. For every $\delta \in (0, 1/2)$, we have $\sup_{1-\delta < t \le 1} \rho(t) = (1-\delta)\delta$, so, keeping in mind that $\tr(E_k * S) = \mu(E_k) \tr(S)$ and that $\int_G h_S d\mu = 1$, we obtain
    \begin{align*}
        0 
            &\le \int_G h_S(x) \left(1 - \frac{\mu(E_k \cap x^{-1}E_k)}{\mu(E_k)}\right) d\mu(x) \\
            &= 1 - \mu(E_k)^{-1} \int_G h_S(x) \mu(E_k \cap x^{-1}E_k) d\mu(x) \\
            &= \frac{\tr(E_k * S) - \tr((E_k * S)^2)}{\tr(E_k * S)} \\
            &= \frac{\tr(\rho(E_k * S))}{\tr(E_k * S)} \\
            &\le 1 - C_\delta^{(k)} (1 - \delta - (1-\delta)\delta) \\
            &= 1 - (1-\delta)^2C_\delta^{(k)},
    \end{align*}
    where we used Lemma \ref{la:traceofsquare} (ii) in the third line and Lemma \ref{la:boundtrrhot} in the fifth. By assumption, the very right hand side converges to $2\delta - \delta^2$ as $k \to \infty$, and $\delta$ can be taken to be arbitrarily small. Thus, if we denote
    \[
        \tilde{\beta}_{E_k}: G \to [0, 1], \; x \mapsto 1 - \frac{\mu(E_k \cap x^{-1}E_k)}{\mu(E_k)}
    \]
    for every $k \in \N$, then
    \begin{equation}\label{eq:inttozero1}
        \limk \int_G h_S \tilde{\beta}_{E_k} d\mu(x) = 0.
    \end{equation}
    As the integrand is non-negative, we can find a subsequence $(\tilde{\beta}_{E_{k_j}})_{j \in \N}$ such that $(h_S \tilde{\beta}_{E_{k_j}})_{j \in \N}$ converges to $0$ pointwise almost everywhere. Since $h_S$ is continuous with $h_S(e) > 0$ by Lemma \ref{la:traceofsquare} (i), there exists a symmetric neighborhood $V$ of $e$ in $G$ such that $h_S$ is strictly positive on $V$. Thus, $(\tilde{\beta}_{E_{k_j}})_{j \in \N}$ converges to $0$ pointwise almost everywhere on $V$. Further, for all $x \in G$ and $j \in \N$, we have
    \[
        1 - \Delta_G(x) = 1 - \frac{\mu(x^{-1}E_{k_j})}{\mu(E_{k_j})} \le \tilde{\beta}_{E_{k_j}}(x).
    \]
    Letting $j \to \infty$ gives $\Delta_G(x) \ge 1$ almost everywhere on $V$, and, by continuity of $\Delta_G$, even everywhere on $V$. But since $V$ is symmetric, we then have $\Delta_G (x) \ge 1$ and $\Delta_G(x)^{-1} = \Delta_G(x^{-1}) \ge 1$ for each $x \in V$, so $(\Delta_G)|_V \equiv 1$. $\Delta_G$ is a homomorphism and $V$ generates the connected group $G$, thus $\Delta_G \equiv 1$ and $G$ is unimodular. With this, we know that
    \[
        \tilde{\beta}_{E_k}(x) = \beki(x^{-1})
    \]
    for all $x \in G$, so we substitute the inverse in \eqref{eq:inttozero1} and use the symmetry of $h_S$ to obtain
    \[
        \limk \int_G h_S \beki d\mu = 0.
    \]
    Now choose $\varepsilon > 0$ and a neighborhood $W$ of $e$ such that $h_S|_W \ge \varepsilon$, hence
    \begin{equation}\label{eq:inttozero2}
        \limk \int_W \beki d\mu \le \limk \frac{1}{\varepsilon} \int_W h_S \beki d\mu = 0.
    \end{equation}
    Further, let $K$ be a compact, symmetric neighborhood of $e$ such that $K^2 \subseteq W$. Reusing the inequality \eqref{eq:uniformbetaineq} from the proof of Proposition \ref{prop:folnerequiv}, we have
    \begin{align*}
        \beki(x) 
            &\le \mu(K)^{-1} \left(\int_K \beki d\mu + \left(\max_{y \in KK^{-1}} \Delta_G(y)\right) \int_{KK^{-1}} \beki d\mu\right) \\
            &\le \mu(K)^{-1} \left(1 + \max_{y \in KK^{-1}} \Delta_G(y)\right) \int_W \beki d\mu
    \end{align*}
    for all $x \in K$, because $K \subseteq KK^{-1} = K^2 \subseteq W$. Thus, $\seqk{E_k^{-1}}$ is pointwise Følner on $K$ by \eqref{eq:inttozero2}, and Proposition \ref{prop:folnerequiv} then asserts that $\seqk{E_k^{-1}}$ is already a Følner sequence.
\end{proof}

In particular, for a connected group $G$ with a square-integrable, irreducible representation, asymptotic eigenvalue behavior as described in (b) is possible if and only if $G$ is unimodular and amenable. This means that the claim \cite[Theorem 5.14]{Halvdansson} for the affine group cannot hold.

\begin{remark}\label{rm:maindisconnected}
    For the sake of an equivalent characterization, we assumed $G$ to be connected. However, the implication (a)$\implies$(b) still holds if this is not the case. Further, if $G$ is only almost connected (i.e. if $G/G_0$ is compact, where $G_0$ is the identity component), then (b) still implies unimodularity, but $\seqk{E_k^{-1}}$ is only guaranteed to be a Følner sequence on $G_0$ by Remark \ref{rm:folnerdisconnected}.
\end{remark}

\begin{remark}\label{rm:mainmodz}
    This section again is compatible with the scenario where $G$ is a simply connected, connected, nilpotent Lie group with a representation that is only square-integrable modulo the center $Z$. The concept of a density operator as introduced in Definition \ref{def:densityop} even becomes significantly easier since $G$ is unimodular; the density operators are precisely the non-negative trace class operators $S$ satisfying $\tr(S) = d^2$, where $d$ is the constant from Theorem \ref{thm:duflomooremodz}. Lemma \ref{la:boundtrrhot} does not mention the group at all and Lemma \ref{la:traceofsquare} only uses properties of operator convolutions and of $\alpha$ that we already showed to be transferrable in Section \ref{sec:opconv}. Finally, note that the proof of Theorem \ref{thm:main} only relies on properties of operator convolutions and the representation implicitly through Lemma \ref{la:traceofsquare}, thus it is immediately possible to apply Theorem \ref{thm:main} to the modified operator convolutions on $G/Z$ discussed in Remark \ref{rm:convmodz}.
\end{remark}

%%%%%%%%%%%%%%%%%%%%%%%%%%%%%%%%%%%%%%%%%%%%%%%%%%%%%%%%%%%%%%%%%%%%%%%
%%%%%%%%%%%%%%%%%%%%%%%%%%%%%%%%%%%%%%%%%%%%%%%%%%%%%%%%%%%%%%%%%%%%%%%

\section{Applications to nilpotent Lie groups}\label{sec:appl}

In this section, we want to turn to a setting where the first implication of Theorem \ref{thm:main} gives a positive result. From here, let $G$ be a simply connected, connected, nilpotent Lie group with center $Z$ and $\pi$ an irreducible representation of $G$ on a Hilbert space $\Hi$ such that $\pi$ is square-integrable modulo $Z$. Let $\dot{\mu}$ denote the Haar measure on the quotient $G/Z$, which is again a simply connected, connected, nilpotent Lie group.

\subsection{Application to word metric balls}\label{sub:applwordmetric}

The quotient $G/Z$ is unimodular and, by Lemma \ref{la:nilpotentfolner}, we know Følner sequences in $G/Z$, thus Theorem \ref{thm:main} is applicable. With Remark \ref{rm:mainmodz} in mind and $d>0$ being the constant from Theorem \ref{thm:duflomooremodz}, we state the result for nilpotent Lie groups explicitly.

\begin{theorem}\label{thm:nilpotentmain}
    Let $S$ be a density operator on $\Hi$ and $V$ a compact, symmetric neighborhood of the identity in $G/Z$. If $\seqn{\lambda_n^{(k)}}$ denote the eigenvalues of the operator $V^k * S$, then
    \[
        \limk \frac{\#\set{n \mid \lambda_n^{(k)} > 1-\delta}}{d^2 \dot{\mu}(V^k)} = 1
    \]
    for every $\delta \in (0, 1)$.
\end{theorem}

Note that inversion of the Følner sequence is not necessary here since $V$ is assumed to be symmetric.

%%%%%%%%%%%%%%%%%%%%%%%%%%%%%%%%%%%%%%%%%%%%%%%%%%%%%%%%%%%%%%%%%%%%%%%

\subsection{Application to dilations}\label{sub:appldilations}

Analogously to the previous subsection, we treat the case where additionally, the quotient $G/Z$ is homogeneous, as discussed in Section \ref{sub:homogeneous}. Let $(\dot{\Gamma}_r)_{r>0}$ denote the dilations on $G/Z$ and $\dot{Q}$ the homogeneous dimension. Then we can apply Theorem \ref{thm:main} to the Følner sequences constructed in Lemma \ref{la:homogeneousfolner} to obtain the following.

\begin{theorem}\label{thm:homogeneousmain}
    Let $S$ be a density operator on $\Hi$, $E \subseteq G/Z$ satisfy $0 < \dot{\mu}(E) < \infty$, and $\seqk{r_k}$ be a sequence of positive real numbers with $\limk r_k \to \infty$. If $\seqn{\lambda_n^{(k)}}$ denote the eigenvalues of the operator $\dot{\Gamma}_{r_k}(E) * S$, then
    \[
        \limk \frac{\#\set{n \mid \lambda_n^{(k)} > 1-\delta}}{d^2 \dot{\mu}(E) r_k^{\dot{Q}}} = 1
    \]
    for every $\delta \in (0, 1)$.
\end{theorem}

Inversion of the Følner sequence is again not necessary since $\seqk{\dot{\Gamma}_{r_k}(E)^{-1}} = \seqk{\dot{\Gamma}_{r_k}(E^{-1})}$ is a Følner sequence as well by Lemma \ref{la:homogeneousfolner}. Also, the discrete limit may be replaced with the continuous $\lim_{r \to \infty}$, as we consider arbitrary real sequences that tend to $\infty$.

At this point, we want to consider, as a special case, the Heisenberg group $\mathbb{H}^1$, which is the the lowest-dimensional non-abelian simply connected, connected, nilpotent Lie group. We define it here as the manifold $\R^3$ with the group law
\[
    (x, y, z) (x', y', z') := (x + x', y + y', z + z' + x y').
\]
The center $Z$ is then given by the subgroup of all elements of the form $(0, 0, z)$, $z \in \R$, so that $\mathbb{H}^1/Z \cong \R^2$ is the 2-dimensional abelian Lie group. In particular, the Haar measure on $\mathbb{H}^1/Z$ is the 2-dimensional Lebesgue measure denoted by $\lambda_2$, and the dilations are given by $\dot{\Gamma}_r (x, y) = (rx, ry)$ for $(x, y) \in G/Z$ and $r > 0$. A suitable representation is the \newterm{Schrödinger representation} $\pi$ of $\mathbb{H}^1$ on the Hilbert space $L^2(\R)$ defined by
\[
    \pi(x, y, z)\psi(u) := e^{2\pi i(z + y u)} \psi(u + x)
\]
for $\psi \in L^2 (\R)$, $(x, y, z) \in \mathbb{H}^1$, and $u \in \R$. This can be shown to be irreducible and square-integrable modulo $Z$ with $d = 1$. Thus, Theorem \ref{thm:homogeneousmain} takes the following form: Let $S \in \mathcal{S}^1(L^2(\R))$ with $\tr(S) = 1$ and $E \subseteq \R^2$ with $0 < \lambda_2(E) < \infty$. If $\seqn{\lambda_n^{(r)}}$ denotes the eigenvalue sequence of the operator $rE * S$, then
\[
    \lim_{r \to \infty} \frac{\#\set{n \mid \lambda_n^{(r)} > 1-\delta}}{\lambda_2(E) r^2} = 1.
\]
This particular fact was also shown in  \cite[Corollary 2.2]{MainResHeisenberg1} using a different, more general approach, however only in the case $S = \vp \otimes \vp$ for $\n{\vp}_{L^2(\R)} = 1$. It was generalized to arbitrary density operators in \cite[Theorem 4.4]{MainResHeisenberg2} using a method similar to ours, restricted to $\mathbb{H}^1$ with the Schrödinger representation.

%%%%%%%%%%%%%%%%%%%%%%%%%%%%%%%%%%%%%%%%%%%%%%%%%%%%%%%%%%%%%%%%%%%%%%%
% END OF CONTENT %
%%%%%%%%%%%%%%%%%%%%%%%%%%%%%%%%%%%%%%%%%%%%%%%%%%%%%%%%%%%%%%%%%%%%%%%

\section*{Data availability}

No data was used for the research described in this article.

\section*{Acknowledgments}

I would like to thank Hartmut Führ for the close guidance and supervision and the helpful suggestions during the writing of this paper and the preceding Master's thesis.

\bibliographystyle{elsarticle-harv}
\bibliography{references}

\end{document}